\documentclass[11pt]{amsart}
\usepackage{amssymb,
}
\usepackage{mathdots}
\usepackage{array}
\usepackage{graphicx}
\usepackage{caption, booktabs}

\usepackage{comment}
\usepackage{cite}

\renewcommand{\d}{\mathrm{d}}

\newcommand{\C}{\mathbb{C}}

\newcommand{\g}{\mathfrak{g}}

\newcommand{\greg}{\mathfrak{g}_{\mathrm{reg}}}

\newcommand{\gsing}{\mathfrak{g}_{\mathrm{sing}}}

\newcommand{\ad}{\mathrm{ad}}

\newtheorem{thm}{Theorem}[section]

\newtheorem{lem}[thm]{Lemma}

\theoremstyle{definition}

\begin{document}

\title[Slodowy slices and Mishchenko-Fomenko subalgebras]
{Slodowy slices and the complete integrability of Mishchenko-Fomenko subalgebras on regular adjoint orbits} 

\author{Peter Crooks}
\author{Stefan Rosemann}
\author{Markus R\"oser}
\address{Institute of Differential Geometry\\ Gottfried Wilhelm Leibniz University Hannover\\ Welfengarten 1 \\ 30167 Hannover\\ Germany}
\email{peter.crooks@math.uni-hannover.de, stefan.rosemann@math.uni-hannover.de, markus.roeser@math.uni-hannover.de}


\date{}

\dedicatory{}

\commby{}

\begin{abstract}
This work is concerned with Mishchenko-Fomenko subalgebras and their restrictions to the adjoint orbits in a finite-dimensional complex semisimple Lie algebra. In this setting, it is known that each Mishchenko-Fomenko subalgebra restricts to a completely integrable system on every orbit in general position. We improve upon this result, showing that each Mishchenko-Fomenko subalgebra yields a completely integrable system on all regular orbits (i.e. orbits of maximal dimension). Our approach incorporates the theory of regular $\mathfrak{sl}_2$-triples and associated Slodowy slices, as developed by Kostant.          
\end{abstract}

\maketitle


\section{Introduction}\label{Section: Introduction}
Mishchenko-Fomenko subalgebras are ubiquitous in both the classical and modern theories of integrable systems, and they give rise to a fruitful synergy of algebraic geometry, Lie theory, and symplectic geometry (see \cite{BolsinovRemarks,BolsinovIzosimov,BolsinovCriterion,Bolsinov,BolsinovOshemkov,Manakov,Mishchenko,KostantHessenberg,PanyushevYakimova}, for instance). One studies them in the context of a finite-dimensional Lie algebra $\mathfrak{g}$, in which case $\mathfrak{g}^*$ carries a canonical Poisson structure. Each regular element $a\in\mathfrak{g}^*$ then determines a Mishchenko-Fomenko subalgebra $\mathcal{F}_a$ of the polynomials on $\mathfrak{g}^*$, obtained by applying an argument-shifting procedure to invariants of the coadjoint representation. This Poisson-commutative subalgebra is often a completely integrable system on $\mathfrak{g}^*$, i.e. $\mathcal{F}_a$ is often complete (see \cite[Theorem 1.3]{Bolsinov}). The most classically studied cases arise when $\mathfrak{g}$ is complex semisimple, in which event every Mishchenko-Fomenko subalgebra is a completely integrable system.  

We now describe the context of interest to us. Let $\mathfrak{g}$ be a complex semisimple Lie algebra of finite dimension $n$ and rank $r$. Note that $\mathfrak{g}$ is the Lie algebra of a connected, simply-connected complex semisimple linear algebraic group $G$. Note also that we have the adjoint representation $$\mathrm{ad}:\mathfrak{g}\rightarrow\mathfrak{gl}(\mathfrak{g}),\quad x\mapsto\mathrm{ad}_x,\quad x\in\mathfrak{g},$$ which satisfies $\mathrm{ad}_x(y)=[x,y]$ for all $x,y\in\mathfrak{g}$. An element $x\in\mathfrak{g}$ is defined to be \textit{regular} if $\dim(\mathrm{ker}(\ad_x))=r$, and we shall let $\greg$ denote the open, dense, $G$-invariant subvariety of all regular elements in $\g$. An adjoint orbit $\mathcal{O}\subseteq\mathfrak{g}$ is defined to be \textit{regular} when $\mathcal{O}\subseteq\mathfrak{g}_{\text{reg}}$, or equivalently $\dim(\mathcal{O})=n-r$. 

Recall that the Killing form induces a $G$-equivariant identification of $\mathfrak{g}$ with $\mathfrak{g}^*$, through which we can transfer relevant algebraic and geometric structures (ex. the Poisson structure, Mishchenko-Fomenko subalgebras) from the latter space to the former. Note that $\mathbb{C}[\mathfrak{g}]:=\mathrm{Sym}(\mathfrak{g}^*)$ is then a Poisson algebra with Poisson centre equal to the subalgebra of all $G$-invariant polynomials, $\mathbb{C}[\mathfrak{g}]^G\subseteq\mathbb{C}[\mathfrak{g}]$. Now fix a regular element $a\in\greg$, and associate to each $f\in\mathbb{C}[\mathfrak{g}]^G$ and $\lambda\in\mathbb{C}$ the polynomial $f_{\lambda,a}\in\mathbb{C}[\mathfrak{g}]$ defined by
$$f_{\lambda,a}(x):=f(x+\lambda a),\quad x\in\g.$$ One then defines $\mathcal{F}_a$ to be the subalgebra of $\mathbb{C}[\mathfrak{g}]$ generated by all polynomials of the form $f_{\lambda,a}$, with $f\in\mathbb{C}[\mathfrak{g}]^G$ and $\lambda\in\mathbb{C}$. 

We wish to study the restriction of a Mishchenko-Fomenko subalgebra $\mathcal{F}_a$ to the symplectic leaves of $\g$, i.e. to the adjoint orbits of $G$. More precisely, let $\mathcal{F}_a\vert_{\mathcal{O}}$ denote the algebra of all functions obtained by restricting elements of $\mathcal{F}_a$ to an adjoint orbit $\mathcal{O}\subseteq\g$. This algebra is Poisson-commutative with respect to the Poisson structure on $\mathcal{O}$, so that it is natural to ask the following question: for which combinations of a regular element $a\in\mathfrak{g}$ and an adjoint orbit $\mathcal{O}\subseteq\mathfrak{g}$ is $\mathcal{F}_{a}\vert_{\mathcal{O}}$ a completely integrable system on $\mathcal{O}$?

There are a few well-known results that address the question posed above. One such result is due to Mishchenko and Fomenko, who proved that $\mathcal{F}_a\vert_{\mathcal{O}}$ is a completely integrable system whenever $a$ and $\mathcal{O}$ are in general position (see \cite[Theorem 4.2]{Mishchenko}). Their proof requires $a$ to be a regular semisimple element and $\mathcal{O}$ to be a regular adjoint orbit, so that the meaning of ``general position'' is stronger than that of ``regular'' for both elements of $\mathfrak{g}$ and adjoint orbits. With this point in mind, there are two notable generalizations of the Mishchenko-Fomenko result. The first follows from Bolsinov's work, and can be stated as follows: if $a\in\mathfrak{g}$ is any regular element, then $\mathcal{F}_a\vert_{\mathcal{O}}$ is a completely integrable system on each adjoint orbit $\mathcal{O}\subseteq\mathfrak{g}$ in general position (see \cite{BolsinovCriterion}). The second generalization derives from Kostant's work, which effectively shows $\mathcal{F}_a\vert_{\mathcal{O}}$ to be a completely integrable system whenever $a$ is a regular semisimple element and $\mathcal{O}$ is a regular adjoint orbit (see \cite[Proposition 4.7]{KostantHessenberg}). 

This paper generalizes all of the above-mentioned results as follows.    

\begin{thm}\label{Thm:MainThm}
If $a\in\mathfrak{g}$ is any regular element, then $\mathcal{F}_a\vert_{\mathcal{O}}$ is a completely integrable system on each regular adjoint orbit $\mathcal{O}\subseteq\mathfrak{g}$.
\end{thm}

\section{Slodowy slices and $\mathfrak{sl}_2$-triples}
Our proof of Theorem \ref{Thm:MainThm} draws from Kostant's work on regular $\mathfrak{sl}_2$-triples and their Slodowy slices, the relevant parts of which we now review. Let all objects and notation be as described in Section \ref{Section: Introduction}, after the first paragraph. We recall that $(\xi,h,\eta)\in\mathfrak{g}^{\oplus 3}$ is called an $\mathfrak{sl}_2$\textit{-triple} if the identities
$$[h,\xi]=2\xi,\quad [h,\eta]=-2\eta,\quad [\xi,\eta]=h$$ hold in $\mathfrak{g}$. In this case, one may consider the associated \textit{Slodowy slice}
$$S(\xi,h,\eta):=\xi+\mathrm{ker}(\mathrm{ad}_{\eta}):=\{\xi+x:x\in\mathrm{ker}(\mathrm{ad}_{\eta})\}\subseteq\mathfrak{g}.$$ We will be particularly interested in those Slodowy slices that arise when $(\xi, h, \eta)$ is a \textit{regular} $\mathfrak{sl}_2$-triple, i.e. when $\xi$ is regular.

\begin{thm}\label{thm: Kostant}\emph{(cf. \cite[Theorem 8]{KostantLie})}
Let $(\xi,h,\eta)$ be a regular $\mathfrak{sl}_2$-triple. If $\mathcal{O}\subseteq\mathfrak{g}$ is any regular adjoint orbit, then $S(\xi,h,\eta)$ intersects $\mathcal{O}$ in a unique point. 
\end{thm} 

We will benefit from constructing a specific regular $\mathfrak{sl}_2$-triple. To this end, let $\mathfrak{b}_{+},\mathfrak{b}_{-}\subseteq\mathfrak{g}$ be opposite Borel subalgebras. Note that $\mathfrak{h}:=\mathfrak{b}_{+}\cap\mathfrak{b}_{-}$ is then a Cartan subalgebra of $\mathfrak{g}$, and that we have roots $\Delta\subseteq\mathfrak{h}^*$. Given $\alpha\in\Delta$, let $\mathfrak{g}_{\alpha}$ denote the $\alpha$-root space in $\mathfrak{g}$, i.e.
$$\mathfrak{g}_{\alpha}:=\{x\in\mathfrak{g}:\mathrm{ad}_y(x)=\alpha(y)x\text{ for all }y\in\mathfrak{h}\}.$$ We may then define collections of positive roots $\Delta_{+}\subseteq\Delta$ and negative roots $\Delta_{-}\subseteq\Delta$ by the condition that $$\mathfrak{b}_{\pm}=\mathfrak{h}\oplus\bigoplus_{\alpha\in\Delta_{\pm}}\mathfrak{g}_{\alpha}$$
as $\mathfrak{h}$-modules. Let $\Pi\subseteq\Delta_{+}$ denote the resulting collection of simple roots. For each $\alpha\in\Pi$, let $h_{\alpha}\in\mathfrak{h}$ be the corresponding simple coroot and choose elements $e_{\alpha}\in\mathfrak{g}_{\alpha}$ and $e_{-\alpha}\in\mathfrak{g}_{-\alpha}$ such that $[e_{\alpha},e_{-\alpha}]=h_{\alpha}$. Let us define $h$ to be the unique element of $\mathfrak{h}$ satisfying $\alpha(h)=-2$ for all $\alpha\in\Pi$. Noting that the simple coroots form a basis of $\mathfrak{h}$, we may write 
$$-h=\sum_{\alpha\in\Pi}c_{\alpha}h_{\alpha}$$
for uniquely determined coefficients $c_{\alpha}\in\mathbb{C}$. Now consider the nilpotent elements of $\mathfrak{g}$ defined by
$$\xi:=\sum_{\alpha\in\Pi}e_{-\alpha} \qquad \text{and} \qquad \eta:=\sum_{\alpha\in\Pi}c_{\alpha}e_{\alpha}.$$ It is then straightrforward to verify that $(\xi,h,\eta)$ is an $\mathfrak{sl}_2$-triple. Since $\xi$ is known to be a regular element (see \cite[Theorem 5.3]{KostantTDS}), we see that $(\xi,h,\eta)$ is regular.  

The following lemma will be needed to help prove the main result of our paper.
\begin{lem}\label{lem: containment}
If $(\xi,h,\eta)$ is the $\mathfrak{sl}_2$-triple constructed above, then $\mathrm{ker}(\mathrm{ad}_{\eta})\subseteq\mathfrak{b}_{+}$. 
\end{lem}

\begin{proof}
Since $\alpha(h)=-2$ for all $\alpha\in\Pi$, it follows that $\mathfrak{b}_{+}$ is the sum of those $\mathrm{ad}_h$-eigenspaces corresponding to non-positive eigenvalues. At the same time, the representation theory of $\mathfrak{sl}_2$ implies that $h$ acts on $\mathrm{ker}(\mathrm{ad}_{\eta})$ with non-positive eigenvalues. We conclude that $\mathrm{ker}(\mathrm{ad}_{\eta})$ is contained in $\mathfrak{b}_{+}$.  
\end{proof}

\section{Proof of Theorem \ref{Thm:MainThm}}
Let all objects be as described in the statement of Theorem \ref{Thm:MainThm} and the paragraphs preceding it. Our objective is to prove that the Poisson-commutative algebra $\mathcal{F}_a\vert_{\mathcal{O}}$ is complete, i.e. that
the subspace $$\d_x(\mathcal{F}_a\vert_{\mathcal{O}}):=\{\d_x f:f\in\mathcal{F}_a\vert_{O}\}\subseteq T_x^*\mathcal{O}$$ has dimension $\frac{1}{2}\dim(\mathcal{O})=\frac{1}{2}(n-r)$ for all $x$ in an open dense subset of $\mathcal{O}$. To this end, consider the set\begin{equation}\label{Equation:Sing}
\mathrm{Sing}_a(\mathcal{O}):=\bigg\{x\in\mathcal{O}:\dim(\d_x(\mathcal{F}_a\vert_{\mathcal{O}}))<\frac{1}{2}(n-r)\bigg\}
\end{equation}
of singularities of $\mathcal{F}_a\vert_{\mathcal{O}}$. Note that $\mathcal{F}_a\vert_{\mathcal{O}}$ is complete if and only if $\mathcal{O}\setminus\mathrm{Sing}_a(\mathcal{O})$ is open and dense in $\mathcal{O}$. At the same time, $\mathcal{O}$ is irreducible and contains $\mathrm{Sing}_a(\mathcal{O})$ as a Zariski-closed subset. This means that either $\mathcal{O}=\mathrm{Sing}_a(\mathcal{O})$ or every irreducible component of 
$\mathrm{Sing}_a(\mathcal{O})$ has strictly positive codimension in $\mathcal{O}$, in which case $\mathcal{O}\setminus\mathrm{Sing}_a(\mathcal{O})$ is necessarily open and dense in $\mathcal{O}$. Accordingly, $\mathcal{O}\setminus\mathrm{Sing}_a(\mathcal{O})$ is open and dense in $\mathcal{O}$ if and only if $\mathcal{O}\neq\mathrm{Sing}_a(\mathcal{O})$. 

In light of the above-mentioned equivalences, we are reduced to proving that $\mathcal{O}\neq\mathrm{Sing}_a(\mathcal{O})$. Let us assume that $\mathcal{O}=\mathrm{Sing}_a(\mathcal{O})$, for the sake of an argument by contradiction. Now recall that $\C[\g]^G$ is generated by $r$ homogeneous, algebraically independent polynomials $f_1,\dots,f_r\in \C[\g]^G$. Set $d_i:=\mathrm{degree}(f_i)$, $i\in\{1,\ldots,r\},$ and recall that $\sum_{i=1}^r d_i = \frac{1}{2}(n+r) =:\ell$ (see \cite[Equation (1)]{Varadarajan}). We define new polynomials $f_{ij}\in \C[\g]$, $i\in\{1,\dots,r\}$, $j\in\{0,\dots,d_i-1\},$ by the following expansions:
$$
f_i(x+\lambda a)=\sum_{j=0}^{d_i-1}f_{ij}(x)\lambda^j+f_i(a)\lambda^{d_i},\quad i\in\{1,\dots,r\},
$$
where $x\in\mathfrak{g}$ and $\lambda\in\mathbb{C}$.
The $\ell$-many polynomials $f_{ij}$ turn out to form a list of algebraically independent generators of $\mathcal{F}_a$ (see \cite[Section 3]{PanyushevYakimova}, for instance). Note also that $f_{i0}=f_i$ for all $i\in\{1,\ldots,r\}$, so that we have introduced only $\ell-r=\frac{1}{2}(n-r)$ new polynomials. Let us record these new polynomials as $f_{r+1},\ldots,f_{\ell}$.

Suppose that $x\in\mathcal{O}$. Since $x\in\mathrm{Sing}_a(\mathcal{O})$, it follows that the $\frac{1}{2}(n-r)$ vectors $\d_x(f_{r+1}\vert_{\mathcal{O}}),\ldots,\d_x(f_{\ell}\vert_{\mathcal{O}})$ are linearly dependent in $T_x^*\mathcal{O}$. This is the statement that $$\sum_{i=r+1}^{\ell}a_i(\d_x(f_{i}\vert_{\mathcal{O}}))=0$$ for some $a_{r+1},\ldots,a_{\ell}\in\mathbb{C}$, not all equal to $0$.
Now note that $f_1,\ldots,f_r$ are constant on $\mathcal{O}$, so that the annihilator of $T_x\mathcal{O}$ in $T_x^*\mathfrak{g}$ must contain $$U:=\mathrm{span}\{\d_x f_1, \ldots, \d_x f_{r}\}.$$ This annihilator is $r$-dimensional, since the regularity of $\mathcal{O}$ implies that $T_x\mathcal{O}$ has dimension $n-r$. At the same time, the fact that $x\in\greg$ implies that $U$ is $r$-dimensional (see \cite[Theorem 7]{KostantLie}). We conclude $U$ is precisely the annihilator of $T_x\mathcal{O}$ in $T_x^*\mathfrak{g}$, so that $$\sum_{i=r+1}^{\ell}a_i(\d_xf_{i})\in U.$$ Given how $U$ is defined, this means that $\d_x f_1, \ldots, \d_x f_{\ell}$ must be linearly dependent in $T_x^*\mathfrak{g}$. Now let $\mathrm{Sing}(\mathcal F_a)$ denote the set of singularities of $\mathcal{F}_a$ in $\mathfrak{g}$, which one defines analogously to \eqref{Equation:Sing}. Since $f_1,\ldots,f_{\ell}$ are algebraically independent generators of $\mathcal{F}_a$, we see that $\mathrm{Sing}(\mathcal{F}_a)$ is the set of points at which $\d f_1,\ldots,\d f_{\ell}$ are linearly dependent. We also have that $\mathrm{Sing}(\mathcal{F}_a)=\gsing+\mathbb{C}\cdot a$ (see \cite[Proposition 3.1]{Bolsinov}), where $\gsing:=\g\setminus\greg$ is the set of singular elements in $\g$. Using these last two sentences, we conclude that $x\in\gsing+\mathbb{C}\cdot a$. In particular, this paragraph establishes the inclusion  
\begin{equation}\label{Equation: Inclusion}\mathcal{O}\subseteq\mathfrak{g}_{\text{sing}}+\mathbb{C}\cdot a.\end{equation}

Now choose opposite Borel subalgebras $\mathfrak{b}_{+},\mathfrak{b}_{-}\subseteq\mathfrak{g}$ with the property that $a\in\mathfrak{b}_{+}$. Beginning with these opposite Borel subalgebras, we may repeat all of the constructions outlined between Theorem \ref{thm: Kostant} and Lemma \ref{lem: containment} to produce a specific $\mathfrak{sl}_2$-triple $(\xi,h,\eta)$. Recall that this triple is regular, so that Theorem \ref{thm: Kostant} implies that the Slodowy slice $S(\xi,h,\eta)$ must intersect $\mathcal{O}$. It now follows from \eqref{Equation: Inclusion} that $S(\xi,h,\eta)$ intersects $\mathfrak{g}_{\text{sing}}+\mathbb{C}\cdot a$. In other words, there exist $x\in\mathfrak{g}_{\text{sing}}$, $\lambda\in\mathbb{C}$, and $y\in\mathrm{ker}(\mathrm{ad}_{\eta})$ such that $x+\lambda a=\xi+y$, i.e.
$$x=\xi+y-\lambda a.$$ Since $y\in\mathfrak{b}_{+}$ (by Lemma \ref{lem: containment}) and $a\in\mathfrak{b}_{+}$, we conclude that $x\in\xi+\mathfrak{b}_{+}.$  At the same time, it is known that $\xi+\mathfrak{b}_{+}\subseteq\mathfrak{g}_{\text{reg}}$ (see \cite[Lemma 10]{KostantLie}). The previous two sentences imply that $x\in\mathfrak{g}_{\text{reg}}$, contradicting the fact that $x\in\mathfrak{g}_{\text{sing}}$. In particular, $\mathcal{O}=\mathrm{Sing}_a(\mathcal{O})$ cannot hold.  This completes the proof.

\subsection*{Acknowledgements}
The authors gratefully acknowledge Alexey Bolsinov for feedback and several valuable discussions. We also wish to thank Anton Izosimov for constructive conversations.

\bibliographystyle{acm} 
\bibliography{MF}

\end{document}